\documentclass[11pt]{article}
\usepackage{amsmath,amssymb,amsthm, latexsym,color,epsfig,enumerate,a4, graphicx}
\usepackage{changepage}
\usepackage{tikz}
\usepackage{xspace}
\usepackage{bbding}
\usepackage[ruled,vlined, linesnumbered]{algorithm2e} 

\theoremstyle{plain}
\newtheorem{theorem}{Theorem}[section]
\newtheorem{lemma}[theorem]{Lemma}
\newtheorem{claim}[theorem]{Claim}

\theoremstyle{definition}

\newcommand{\calH}{\ensuremath{\mathcal H}}

\date{}
\title{\vspace{-0.7cm}Optimal 
induced universal graphs for bounded-degree graphs}
\author{
	Noga Alon 
		\thanks{
		Sackler School of Mathematics and Blavatnik School of Computer Science, Tel Aviv University, Tel Aviv 69978, Israel. Email: {\tt nogaa@tau.ac.il}. 
		Research supported in part by a USA-Israeli BSF grant 2012/107, by an ISF grant 620/13 and by the Israeli I-Core program.
		}
	\and
	Rajko Nenadov
		\thanks{
		School of Mathematical Sciences, Monash University, Melbourne, Australia. Email: {\tt rajko.nenadov@monash.edu}.
		} 
}

\begin{document}
\maketitle

\begin{abstract}

We show that for any constant $\Delta \ge 2$, there exists a graph $G$
with $O(n^{\Delta / 2})$ vertices which contains every $n$-vertex graph
with maximum degree $\Delta$ as an induced subgraph. For odd
$\Delta$ this significantly improves the best-known earlier bound of
Esperet et al. and is optimal up to a constant
factor, as it is known that any such graph must have at least
 $\Omega(n^{\Delta/2})$
vertices.

Our proof builds on the approach of Alon and Capalbo (SODA 2008) together
with several additional ingredients. The construction of $G$ is explicit
and is based on an appropriately defined  
composition of high-girth expander graphs. The proof also provides
an efficient deterministic procedure for finding, for any given 
input graph $H$ on $n$ vertices with maximum degree at most
$\Delta$,
an induced subgraph of
$G$ isomorphic to $H$.

\end{abstract}


\section{Introduction}

Given a finite family of graphs $\mathcal{H}$, a graph $G$ is
\emph{induced universal} for $\calH$ if for every $H \in \mathcal{H}$
it contains an induced subgraph isomorphic to $H$. This notion was
introduced by Rado \cite{rado1964universal} in the 1960s. Observe
that the number of induced subgraphs of $G$ of a certain size
depends only on the number of vertices of $G$, thus the problem of
determining the smallest possible number of vertices $g_v(\calH)$
of an $\calH$-induced-universal graph arises naturally. Indeed,
this problem
has received a considerable amount of 
attention in the past decades for various
families of graphs, including the family of all graphs with $n$ vertices
\cite{alonasymptotically,alstrup2015adjacency,moon1965minimal},
forests and graphs with bounded arboricity
\cite{alstrup2015optimal,alstrup2002small}, bounded-degree
graphs \cite{butler2009induced,esperet2008induced}, planar
graphs \cite{chung1990universal,gavoille2007shorter}, and more. (See
\cite{alstrup2015adjacency} for a detailed summary of the 
known results). We 
briefly mention
some of the highlights.

Possibly the most basic family is the family $\calH(n)$ of all graphs
on $n$ vertices. The problem of estimating $g_v(\calH(n))$, first
studied by Moon \cite{moon1965minimal} and mentioned by 
Vizing \cite{vizing1968some}
in 1968, has been investigated over the years in several papers. 
In a recent work of
Alstrup, Kaplan, Thorup and Zwick \cite{alstrup2015adjacency}
the authors determined this function up to a constant factor,
showing it is at most 
$16 \cdot 2^{n/2}$. (As observed by Moon \cite{moon1965minimal}, a
simple counting argument implies that $2^{(n-1)/2}$ is a lower
bound).
Very
recently, the first author 
\cite{alonasymptotically} further improved the upper
bound  on $g_v(\calH(n))$  to $(1 + o(1))2^{(n - 1)/2}$, thus matching
the lower bound up to a lower order additive term. 
Another family of graphs
which has been studied extensively is the family of all
trees with $n$ vertices. This was first considered in the work of Kannan,
Naor and Rudich \cite{kannan1992implicat}, and subsequent improvements
were obtained by Chung \cite{chung1990universal} and by Alstrup and Rauhe
\cite{alstrup2002small}. Finally, a tight bound of $\Theta(n)$ was recently
proven by Alstrup, Dahlgaard and Knudsen \cite{alstrup2015optimal}.

In the present paper we consider the family of bounded-degree 
graphs on $n$
vertices. Given $\Delta \ge 2$ and $n \in \mathbb{N}$,
let $\calH(n, \Delta)$ denote the family of all graphs on $n$ vertices
with maximum degree at most $\Delta$. One should think of $\Delta$ being
a constant and $n$ being an arbitrary (large) number. 
This family has been studied in several papers that discuss
the smallest possible number of vertices of an induced universal
graph and the smallest possible number of edges in a
{\em universal graph}.
(A graph is universal
for a family $\calH$ if it contains every $H \in \calH$ as a subgraph
(not necessarily induced)). The latter 
was studied in a series of papers \cite{alon2007sparse,alon2001near} culminating in the work of the
first author and 
Capalbo \cite{alon2008optimal} where it is shown that this minimum
is $\Theta(n^{2-2/\Delta}).$

Induced universal graphs for bounded-degree graphs were first studied
by Butler \cite{butler2009induced}. Using a simple counting argument
he observed that
\begin{equation} \label{eq:bound_gv}
	g_v(\calH(n, \Delta)) \ge c n^{\Delta / 2},
\end{equation}
for some constant $c = c(\Delta) > 0$ depending only on
$\Delta$. His main result is that this is indeed the right order
of magnitude in the case where $\Delta$ is even. For odd $\Delta$ 
his result only gives the upper bound 
$g_v(\calH(n, \Delta)) = O(n^{\Delta/2 + 1/2})$,
which simply follows from the bound obtained for the family $\calH(n,
\Delta + 1)$. Using the reduction of Chung \cite{chung1990universal}
which connects universal and induced universal graphs together with
the
sparse universal graphs from \cite{alon2008optimal}, Esperet,
Labourel and Ochem \cite{esperet2008induced} improved the bound for
odd $\Delta$ to $O(n^{\Delta/2 + 1/2 - 1/\Delta})$ and mentioned
the natural problem of closing the gap between the upper and lower
bounds. Here we settle
this problem by giving a construction which matches the lower
bound \eqref{eq:bound_gv}
up to a constant factor.

\begin{theorem} \label{thm:main}
For every integer $\Delta \ge 2$ there exists a constant $c_\Delta >
0$, such that for every $n \in \mathbb{N}$ there is an $\mathcal{H}(n,
\Delta)$-induced-universal graph $\Gamma$ with at most $c_\Delta n^{\Delta
/ 2}$ vertices.
\end{theorem}

The construction of $\Gamma$ is explicit and the proof supplies
a polynomial
time deterministic procedure for finding, for any given 
$H \in \calH(n, \Delta)$, an induced subgraph of $\Gamma$
isomorphic to $H$.

The rest of the 
paper is organised as follows. In the next section we discuss the
main challenges for the case of odd $\Delta$ and present a 
rough overview of the
construction and proof. In Section \ref{sec:preliminaries} 
we introduce the main building
block in our construction, high-girth expander graphs, and state some
of their properties. In Section \ref{sec:decomposition}
we state the decomposition result from \cite{alon2007sparse},
an analogue of Petersen's theorem for $(2k+1)$-regular graphs. The
construction of $\Gamma$ is given in Section \ref{sec:construction}
and in Section \ref{sec:proof} we prove that $\Gamma$ is indeed induced
universal. Finally, Section \ref{sec:algo} briefly summarises the 
algorithmic
aspects of the proof. Throughout the paper we make no attempts to optimise
the constants.

\medskip
\noindent
\textbf{Notation. }  Given a graph $G$ and an integer $k \ge 1$, the
$k$-th power $G^k$ of $G$ is the graph on the same vertex set as $G$
where two distinct vertices are adjacent if they are at distance at most
$k$ in $G$. We say that a sequence of vertices $(v_1, v_2, \ldots, v_k)$,
$v_i \in G$, forms a \emph{walk} if $\{v_i, v_{i+1}\} \in G$ for every
$i < k$. If additionally no two vertices in the sequence are the same,
it is a \emph{path}. For two graphs $G$ and $H$, we say that
a mapping $f \colon H \rightarrow G$ of the vertices of $H$ into 
the vertices of
$G$ is a \emph{homomorphism} if $\{v, w\} \in H$ implies $\{f(v), f(w)\}
\in G$. If $f$ is injective we say that it is an \emph{embedding}, and
if furthermore $\{v, w\} \in H$ iff $\{f(v), f(w)\} \in G$ then we say
it is an \emph{induced embedding}.

Finally, let $\mathcal{P}(S)$ denote the powerset of a 
finite set $S$ (i.e. the
family of all subsets of $S$), and put $[n] := \{0, \ldots, n\}$ for $n
\in \mathbb{N}$.

\section{Overview of the proof} \label{sec:overview}

In order to demonstrate the main ideas and challenges in our work 
it is instructive
to first review the approach of Butler \cite{butler2009induced}
which determines $g_v(\calH(n, \Delta))$ up to the constant factor
for even $\Delta$. At the heart of his proof lies the classical
decomposition result of Petersen (see \cite{lovasz2009matching}), which
states that every $(2k)$-regular graph can be decomposed into $k$
edge-disjoint $2$-regular graphs. Since each $2$-regular graph is a
collection of cycles, it is not too difficult to construct a graph $F$
with $O(n)$ vertices which is induced universal for such graphs with
$n$ vertices. Now apply the idea of Chung \cite{chung1990universal}
(implicit already in \cite{kannan1992implicat}):
a graph $\Gamma$ is defined on the vertex set $(V(F))^{\Delta/2}$ and
two vertices $\mathbf{x} = (x_1, \ldots, x_{\Delta/2})$ and $\mathbf{y}
= (y_1, \ldots, y_{\Delta/2})$ are adjacent iff they are adjacent
in
at least one coordinate (in graph $F$). Such a graph $\Gamma$
has $O(n^{\Delta/2})$ vertices. An induced embedding of $H \in
\mathcal{H}(n, \Delta)$ is obtained by decomposing $H$ into
$\Delta/2$ $2$-regular
subgraphs and by embedding 
each subgraph into the copy of $F$ corresponding to a separate
coordinate.

However, this strategy does not work if $\Delta$ is odd. Indeed, in
this case one
cannot even have $\Delta/2$ coordinates. A simple solution, inspired by
the work of Alon and Capalbo \cite{alon2008optimal}, is to double the
number of coordinates: instead of having each coordinate of $\Gamma$
correspond to a graph of size $O(n)$ and be responsible for the
existence of an edge, we let each coordinate correspond to a graph
$F'$ of size $O(\sqrt{n})$ and be responsible for \emph{half} an
edge. In other words, graph $\Gamma$ has $\Delta$ coordinates and in
order to have an edge between $\mathbf{x}$ and $\mathbf{y}$ we require
that they are adjacent in at least two of them. This gives  a graph
$\Gamma$ with the desired number of vertices $O(n^{\Delta/2})$, 
overcoming the divisibility
issue. An embedding of $H$ into $\Gamma$ is defined as before
(with respect to the new decomposition, which we present shortly).

Several problems in this approach 
are evident. First, the decomposition result of
Petersen no longer holds. Instead, we use the decomposition result
of \cite{alon2007sparse} which states that every $\Delta$-regular
graph can be decomposed into $\Delta$ subgraphs (rather than $\Delta/2$,
as in Petersen's theorem) such that each is a collection of `almost'
cycles and each edge of $H$ belongs to exactly two of them. Note that
this is exactly what we need since the endpoints of each edge in $H$
have to be adjacent in at least two coordinates.

Second, we cannot guarantee that the embedding of each such subgraph
into $F'$ is induced simply because the number of vertices of $F'$
is too small. In fact, we have to take a \emph{homomorphism} of
each subgraph into $F'$, rather than an injective embedding, which
potentially may create halves of some undesired edges. Our aim is to show
that by choosing homomorphisms of other subgraphs into $F'$ carefully
we can simultaneously avoid creating the other half of any such undesired
edge. In order to do that, graph $F'$ has to allow enough flexibility
and a simple construction such as the one from \cite{butler2009induced}
does not suffice. It turns out that taking $F'$ to be a high-girth
expander is helpful here.
However, even with such a graph $F'$ we are
only able to show that most of the undesired edges will have at most
half of an edge present in $\Gamma$, but a few will slip past. Finally, we
take care of those by introducing additional `layers' of coordinates
in $\Gamma$ which correspond to constant size structures (with each
layer dealing with a different type of the remaining undesired
edges). 
These do not have a significant impact on the number of vertices as they have a constant size.
The precise details require some careful analysis, as described in
the subsequent sections.

\section{Preliminaries} \label{sec:preliminaries}

\subsection{High-girth expander graphs}

In this section we describe the main building block in our construction, the so-called \emph{Ramanujan} graphs, and state the main lemma (Lemma \ref{lemma:f_walk}) used to find an induced embedding of a graph $H$.

A $d$-regular graph $F$ is a \emph{Ramanujan} graph if all non-trivial eigenvalues have absolute value at most $2 \sqrt{d - 1}$. It is well-known that non-bipartite Ramanujan graphs are good expanders (e.g. see Lemma 9.2.4 in \cite{alon2004probabilistic}) and there are known explicit constructions of such graphs with high girth (see \cite{lubotzky1988ramanujan,margulis1988explicit}). It should be noted that any high-girth expander with constant maximum degree would serve our purpose, thus the latter is the reason why we opt for Ramanujan graphs. See \cite{davidoff2003elementary} for a self-contained account on Ramanujan graphs and the necessary background. The following statement is a simplified version of the theorem of Lubotzky, Phillips and Sarnak \cite{lubotzky1988ramanujan}. 


\begin{theorem} \label{thm:Ramanujan}
	Let $p$ and $q > 2 \sqrt{p}$ be primes congruent to $1$ modulo $4$, such that $p$ is a quadratic residue modulo $q$. Then there is an explicit construction of a non-bipartite $(p+1)$-regular Ramanujan graph with $m  = q(q^2 - 1)/2$ vertices and girth at least $\frac{1}{2} \log_{p} m$.
\end{theorem}

 Next, we discuss expansion properties of graphs given by Theorem \ref{thm:Ramanujan}.

\subsubsection{Expansion properties of Ramanujan graphs}

Given a graph $F$ with $\ell$ vertices and a collection of subsets $S_0, \ldots, S_{q-1}$, for some $q \in \mathbb{N}$, we say that a vertex $v \in V(F)$ is \emph{$q$-expanding} with respect to the sets $S_i$ if the following holds: for any path $P$ in $F$ of size at most $q$ (i.e. $P$ has at most $q$ vertices), there are at least $\ell/2$ vertices $w_{q - 1} \in V(F)$ such that there exists a path $(v, w_0, . . . , w_{q - 1})$ in $F$ with $w_i \notin S_i \cup P$ for every $0 \le i \le q - 1$. The following lemma shows that if $S_i$ is not too large, there are many $q$-expanding vertices.

\begin{lemma}[Lemma 4.5, \cite{alon2008optimal}] \label{lemma:expand}
Suppose $d > 720$ and let $F$ be a non-bipartite $d$-regular Ramanujan graph with $\ell$ vertices and girth at least $\frac{1}{2} \log_{d-1} \ell$. Then, for $q = \lceil \log \ell / \log 10 \rceil$ and every collection of subsets $S_0, \ldots, S_{q-1} \subseteq V(F)$ of size $|S_i| \le \ell/20$ for $0 \le i \le q - 1$, there are at least $2\ell/3$ $q$-expanding vertices.
\end{lemma}

By repeated application of Lemma \ref{lemma:expand}, we show that such graph $F$ contains arbitrarily long walk which is well distributed, certain pairs of vertices are not close,  and `locally' forms a path. The following statement makes this precise.

\begin{lemma} \label{lemma:f_walk}
	Let $F$ be as stated in Lemma \ref{lemma:expand}, and set $q = \lceil \log \ell / \log 10 \rceil$. Then, for every $n \in \mathbb{N}$ and a function $\sigma \colon [n-1] \rightarrow \mathcal{P}([n-1])$ (recall that $[n-1] = \{0, \ldots, n - 1\}$) such that
	\begin{itemize}
		\item $\sigma(t) \subseteq \{0, \ldots, (\lfloor t / q \rfloor - 1) \cdot q - 1\}$, and
		\item $|\sigma(t)| \le \ell / (160 d^4)$,
	\end{itemize}
	for all $0 \le t \le n - 1$, there exists a mapping $f \colon [n-1] \rightarrow V(F)$ with the following properties: 
	\begin{enumerate}[(F1)]
		\item $|f^{-1}(v)| \le 40 \lceil n / \ell \rceil$ for every $v \in V(F)$, \label{pr:F1}
		\item $f(t) \neq f(t')$ and the distance between $f(t)$ and $f(t')$ is at least $5$, for every $t \in [n-1]$ and $t' \in \sigma(t)$, and \label{pr:F2}
		\item the sequence 
		$$
			(f(k q), f(k q + 1), \ldots, f(kq + \hat q - 1))			
		$$
		forms a path in $F$ for every $0 \le k < n / q$ and $\hat q = \min \{ 2q, n - kq\}$. \label{pr:F3}
	\end{enumerate}
\end{lemma}

\noindent
We remark that there is a trade-off between the size of $\sigma(t)$ and the distance in  the property (F2). In particular, by requiring that $\sigma(t)$ is smaller we could achieve larger distance. From the property (F3) we conclude that the sequence
$$
	(f(0), f(1), \ldots, f(n-1))
$$
forms a walk in $F$.

\begin{proof}
For simplicity, let us assume $n$ is divisible by $q$. We inductively define 
$$
	f_k \colon \{0, \ldots, kq - 1\} \rightarrow V(F)
$$
for $1 \le k \le n / q$, such that each $f_k$ satisfies (F1)--(F3), i.e.
\begin{enumerate}[(i)]
	\item $|f^{-1}(v)| \le 40 \lceil n / \ell \rceil$ for every $v \in V(F)$,
	\item $f(t) \neq f(t')$ and the distance between $f(t)$ and $f(t')$ is at least $5$ for every $t \in [kq - 1]$ and $t' \in \sigma(t)$, and 
	\item the sequence 
		$$
			(f(k' q), f(k' q + 1), \ldots, f(k'q + \hat q - 1))			
		$$
		forms a path in $F$ for every $0 \le k' \le k - 1$ and $\hat q = \min \{ 2q, kq - k'q\}$.
\end{enumerate}
Note that then $f := f_{n/q}$ satisfies the properties of the lemma. Moreover, for each $1 \le k < n / q$ we shall further assume that $f_k(kq - 1)$ is $q$-expanding with respect to the sets $S_k(i) = D_k(i) \cup A_k$ ($0 \le i \le q - 1$), where
\begin{align*}
	D_k(i) &= \{v \in V(F) \; : \; \exists t' \in \sigma(kq + i) \; \text{ such that } \; \{f_k(t'), v\} \in F^4 \; \text{ or } \; f_k(t') = v\}, \\
	A_k &= \{v \in V(F) \; : \; | f_{k-1}^{-1}(v) | = 40 \lceil n / \ell \rceil \}.
\end{align*}
Let $f_0$ to be an empty function, and $A_{n/q} = D_{n/q}(i) = \emptyset$ for $0 \le i \le q - 1$.

We start by constructing $f_1$. Let $X_1 \subseteq V(F)$ be the set of $q$-expanding vertices with respect to empty sets (i.e. $S_i = \emptyset$ for $0 \le i \le q - 1$). From Lemma \ref{lemma:expand} we get $|X_1| \ge 2\ell / 3$. Pick an arbitrary vertex $v \in X_1$. By the definition, there are at least $\ell / 2$ vertices $w_{q-1} \in V(F)$ for which there exists a path $(v, w_0, \ldots, w_{q-1})$ in $F$ and from $|X_1 \setminus \{v\}| \ge \ell / 2$ we deduce that at least one such vertex $w_{q - 1}$ belongs to $X_1$. Let $(v, w_0, \ldots, w_{q-1})$ be the corresponding path and set $f_1(i) := w_i$ for $0 \le i \le q - 1$. We now verify that $f_1$ has the desired properties. First, since $(f_1(0), \ldots, f_1(q-1))$ is a path we have $|f_1^{-1}(v)| \le 1$ for every $v \in V(F)$ and, moreover, (iii) is satisfied. Next, observe that $\sigma(i) = \emptyset$ for $i \in [2q - 1]$ and therefore the property (ii) is trivially satisfied. Finally, this also implies that the sets $S_1(i)$ are empty and from $f_1(q-1) \in X_1$ we conclude that $f_1(q-1)$ is $q$-expanding with respect to the sets $S_1(i)$.

Let us assume we have defined such functions $f_1, \ldots, f_k$, for some $1 \le k < n / q$. We aim to construct $f_{k+1}$. Recall that $f_k \colon \{0, \ldots, kq - 1\} \rightarrow V(F)$, and let us set $f_{k+1}(i) = f_k(i)$ for $0 \le i \le k q - 1$. If $k + 1 = n/q$ then $A_{k+1}$ and $D_{k+1}(i)$ are empty sets. Otherwise, for every $0 \le i \le q - 1$ and $t' \in \sigma((k + 1)q + i)$ we have $t' \le kq - 1$. Importantly, this implies that $D_{k+1}(i)$ is well defined at this point, that is, it does not depend on how we define $f_{k+1}(kq), \ldots, f_{k+1}(kq + q - 1)$. Since $F$ is $d$-regular, there are at most $4d^4$ vertices at distance less than 5 from any vertex $v \in V(F)$, including the vertex $v$ itself. In particular, we have 
$$
	|D_{k+1}(i)| \le 4d^4 |\sigma((k+1)q + i)| \le \ell / 40.
$$
Observe also that
$$
	|A_{k+1}| \le \frac{n}{40 \lceil n / \ell \rceil} \le \ell / 40.
$$
Therefore, we can apply Lemma \ref{lemma:expand} with $S_i = S_{k+1}(i)$ and let $X_{k+1} \subseteq V(F)$ be the set of $q$-expanding vertices with respect to these sets. From the property (iii) of $f_k$ we have that
$$
	P_{k-1} = (f_k((k-1)q), \ldots, f_k(kq - 1))
$$
forms a path in $F$. Since $|X_{k+1}| \ge 2\ell / 3$ and $f_k(kq - 1)$ is $q$-expanding with respect to the sets $S_k(i)$, there exists a vertex $w_{q-1} \in X_{k+1}$ and a path 
$$
	P_k = (f_k(kq - 1), w_0, \ldots, w_{q-1})
$$
such that
\begin{equation} \label{eq:w_i_avoid}
	w_i \notin S_k(i) \cup P_{k-1} = D_k(i) \cup A_k \cup P_{k-1}
\end{equation}
for every $0 \le i \le q - 1$. We claim that $f_{k+1}(kq + i) := w_i$ satisfies the required properties.

To verify the property (i), it suffices to only consider a vertex
$$
	v \in \{f_{k+1}(kq), \ldots, f_{k+1}(kq + q - 1)\},
$$ 
since for all other vertices this property is inherited from $f_k$. From \eqref{eq:w_i_avoid} we have $v \notin P_{k-1}$ and therefore
$$
	f_k^{-1}(v) = f_{k-1}^{-1}(v).
$$
Furthermore, from $v \notin A_{k}$ and the assumption that $f_{k-1}$ satisfies the property (i) we infer 
$$
	|f_{k - 1}^{-1}(v)| < 40 \lceil n / \ell \rceil.
$$
Finally, since $P_k$ is a path there exists exactly one $0 \le i \le q - 1$ such that $f_{k+1}(kq + i) = v$. Together with the previous observations, this shows $|f_{k+1}^{-1}| \le 40 \lceil n / \ell \rceil$.

Similarly, to verify the property (ii) it is enough to only consider $f(t)$ and $f(t')$ for $t = kq + i$ and $t' \in \sigma(t)$, for some $0 \le i \le q - 1$. However, this follows immediately from the definition of $D_{k}(i)$ and \eqref{eq:w_i_avoid}. Next, the property (iii) for $k' = k - 1$ follows from \eqref{eq:w_i_avoid} and the fact that $P_{k-1}$ and $P_k$ are paths with a common endpoint. For smaller values of $k'$ the property (iii) is, again, inherited from $f_k$. Finally, $f_{k+1}((k+1)q - 1) = w_{q-1} \in X_{k+1}$ is $q$-expanding with respect to the sets $S_{k+1}(i)$ by the construction. This finishes the proof of the lemma.
\end{proof}

\subsection{Graph-decomposition result} \label{sec:decomposition}

We use the following graph-decomposition result proven in \cite{alon2007sparse}. An \emph{augmentation} of a graph $T$ is any graph obtained from $T$ by choosing an arbitrary subset of vertices $U \subseteq V(T)$ and adding a matching between $U$ and a new set $U'$ of $|U'| = |U|$ vertices. We call a graph \emph{thin} if it has maximum degree at most $3$ and every connected component is either an augmentation of a cycle or a path, or it has at most two vertices of degree $3$.

\begin{theorem}[\cite{alon2007sparse}] \label{thm:decompose}
Let $\Delta \ge 2$ be an integer and let $H$ be a graph with maximum degree at most $\Delta$. Then there exist spanning subgraphs $H_1, \ldots, H_\Delta \subseteq H$ such that each $H_i$ is thin and every edge of $H$ lies in precisely two graphs $H_i$.
\end{theorem}

In the case where $\Delta$ is odd, Theorem \ref{thm:decompose} can be seen as a generalization of Petersen's theorem. It was observed in \cite{alon2007sparse} that every thin graph with at most $n$ vertices is a subgraph of $P^4_n$, the $4$-th power of a path with $n$ vertices.

\section{The construction} \label{sec:construction}

For the rest of the paper, let
$$
	d = 734, \quad z = 160 d^5 \quad \text{ and } \quad m = m(n) = 5 \cdot 160 \Delta d^8 \sqrt{n}.
$$
Note that $d$ is chosen such that we can apply Theorem \ref{thm:Ramanujan} with $p = d - 1$ and Lemma \ref{lemma:f_walk}. Our construction relies on the existence of high-girth $d$-regular Ramanujan graphs $R_z$ and $R_m$, with $|V(R_z)| \ge z$ being a constant and $m \le |V(R_m)| \le 32 m$. In particular, $R_z$ is obtained by applying Theorem \ref{thm:Ramanujan} with the smallest $q > 2z^{1/3}$ such that $p = d- 1$ and $q$ satisfy the required conditions. Note that the size of $R_z$ does not depend on $\Delta$, and as such is used for every family $\calH(n, \Delta)$. On the other hand, it follows from the distribution of primes in arithmetic progressions that there exists a prime $q \in (2m^{1/3}, 4m^{1/3})$ such that $q$ is congruent to $1$ modulo $4(d - 1)$, provided $m = m(n)$ is sufficiently large. It is easy to see that $p = d - 1$ is a quadratic residue modulo such $q$ and both numbers are congruent to $1$ modulo $4$. Therefore by Theorem \ref{thm:Ramanujan} there exists an explicit construction of a desired graph $R_m$ with at most $32m$ vertices.

Note that the degree of every vertex in $R_m^4$, the $4$-th power of $R_m$, is at most $4d^4$. Let $\rho_v \colon N_{R_m^4}(v) \rightarrow [4d^4]$ be an arbitrary ordering of neighbours of $v$ in $R_m^4$, for every $v \in V(R_m^4)$. We define $\Gamma = \Gamma(\Delta, n)$ to be the graph on the vertex set 
$$
	V(\Gamma) = V(R_m) \times \left( V(R_m) \times \mathcal{P}([4d^4]) \times V(R_z) \right)^{\Delta-1},
$$
and vertices $(x_1, x_2, X_2, u_2, \ldots, x_\Delta, X_\Delta, u_\Delta)$ and $(y_1, y_2, Y_2, w_2, \ldots, y_\Delta, Y_\Delta, w_\Delta)$ are adjacent iff there exist two indices $1 \le j < i \le \Delta$ such that 
\begin{enumerate}[(E1)]
	\item $\{x_j, y_j\}, \{x_i, y_i\} \in R^4_m$, 
	\item $\rho_{x_i}(y_i) \in X_i$ and $\rho_{y_i}(x_i) \in Y_i$, and
	\item $\{u_i, w_i\} \in R_z^4$. 
\end{enumerate}
We leave the discussion on the choice of parameter $m$ and the structure of $\Gamma$ until the next section. Note that $\Gamma$ has $c_\Delta n^{\Delta/2}$ vertices where $c_\Delta > 0$ depends only on $\Delta$, as required. 


\section{Proof of Theorem \ref{thm:main}} \label{sec:proof}

Consider a graph $H \in \mathcal{H}(n, \Delta)$, for some $\Delta \ge 2$ and $n$ sufficiently large. Using the property of high-girth Ramanujan graphs (Lemma \ref{lemma:f_walk}), we show that there exists an induced embedding of $H$ in $\Gamma = \Gamma(\Delta, n)$. Moreover, we give a deterministic strategy how to find such an embedding.

Let $H_1, \ldots, H_\Delta \subseteq H$ be subgraphs given by Theorem \ref{thm:decompose}. As mentioned in Section \ref{sec:decomposition}, for each $1 \le i \le \Delta$ there exists an embedding $\phi_i \colon H_i \rightarrow P_n^4$ of $H_i$ into the $4$-th power of a path with $n$ vertices. For the rest of the proof we identify $V(P_n^4)$ with $[n-1] = \{0, \ldots, n- 1\}$, in the natural order. Our plan is to  construct homomorphisms $f_i \colon H_i \rightarrow R_m^4$ (for $1 \le i \le \Delta$) and $r_i \colon H_i \rightarrow R_z^4$ (for $2 \le i \le \Delta$) such that the following holds for $2 \le i \le \Delta$:
\begin{enumerate}[(H1)]		
	\item if $f_1(h) = f_1(h')$ then $f_i(h) \neq f_i(h')$, for any distinct $h, h' \in V(H)$,
	\item if $|\phi_i(h) - \phi_i(h')| \le 8$ then  $f_i(h) \neq f_i(h')$, for any distinct $h, h' \in V(H)$,
	\item for each $h \in V(H)$, the set
	\begin{align*}		
		B_i(h) = \{ h' \in V(H) \; : \; & \{h, h'\} \notin H, \; |\phi_i(h) - \phi_i(h')| > 4 \; \text{ and } \\
		&\exists j < i \; \text{ such that } \; \{f_j(h), f_j(h')\}, \{f_i(h), f_i(h')\} \in R^4_m \}
	\end{align*}
	is of size at most $d$ and $|\phi_i(h) - \phi_i(h')| > 2z$ for every $h' \in B_i(h)$,
	\item $\{r_i(h), r_i(h')\} \notin R_z^4$ for every $h' \in B_i(h)$.	
\end{enumerate}
Having such homomorphisms, we define $\gamma \colon H \rightarrow \Gamma$ as
$$
	\gamma(h) = \left( f_1(h), f_2(h), \Phi_2(h), r_2(h), \ldots, f_\Delta(h), \Phi_\Delta(h), r_\Delta(h) \right),
$$
where $\Phi_i(h) = \rho_{f_i(h)}(f_i(N_{H_i}(h)))$ is the set of images of neighbours of $h$ in $H_i$ (or, more precisely, the labels associated with these vertices from the point of view of $f_i(h)$). 

Before we prove that $\gamma$ is an induced embedding of $H$ and that such homomorphisms $f_i$ and $r_i$ exist, we briefly spell out properties (H1)--(H4) and discuss roles of different components of $\Gamma$. We say that an edge of $\Gamma$ is \emph{undesired} if it violates the property that $\gamma$ is induced. For brevity, by the \emph{distance} between two vertices of $H$ in $P_n$ we mean the distance of their images in $P_n$ under some $\phi_i(\cdot)$. First, components of $\Gamma$ associated with $R_m^4$ form the `backbone' of our embedding: the property (H1) ensure that $\gamma$ is injective and the property (H3) further restricts potential undesired edges to be spanned by images of those vertices from $H$ which are either very close or very far apart in $P_n$. This already gives some control over the undesired edges, compared to the construction of Alon and Capalbo \cite{alon2008optimal}, and is achieved by increasing the size of $R_m$ significantly (in particular, from $n^{1/\Delta}$ in \cite{alon2008optimal} to $\sqrt{n}$ here). We then take care of the undesired edges between vertices which are close in $P_n$ using components associated with $\mathcal{P}([4d^4])$: from the property (H2) we have that every $f_i$ is `locally' injective (i.e. no two vertices which are close in $P_n$ are mapped to the same vertex), which together with the choice of $\Phi_i(\cdot)$ and the condition (E2) excludes the possibility that there exists an undesired edge between images of such vertices. Finally, the undesired edges between images of vertices which are far apart in $P_n$ are taken care of by mapping them to non-adjacent vertices in $R_z^4$ (property (H4)). Importantly, the property (H3) also guarantees that for every vertex $h \in V(H)$ there are only constantly many vertices $h' \in V(H)$ such that $\{h, h'\}$ is potentially an undesired edge, which allows us to also take $z$ to be a constant. We make this precise in the rest of the proof.

\begin{claim}
$\gamma$ is an induced embedding.
\end{claim}
\begin{proof}
Since each edge $\{h, h'\} \in H$ belongs to exactly two graphs $H_j, H_i$ ($j < i$) and $f_j$, $f_i$ and $r_i$ are homomorphisms, conditions (E1) and (E3) from the definition of $\Gamma$ are satisfied for this particular choice of $j,i$. From $\rho_{f_i(h)}(f_i(h')) \in \Phi_i(h)$ (and similarly the other way around) we have that (E2) holds as well. This implies that $\gamma$ is a homomorphism of $H$ into $\Gamma$. From the property (H1) we infer that $\gamma$ is injective, thus it is an embedding.

Let us assume, towards a contradiction, that there exists $h, h' \in V(H)$ such that $\{h, h'\} \notin H$ and $\{\gamma(h), \gamma(h')\} \in \Gamma$. Let $j, i \in [\Delta]$ be witnesses for $\{\gamma(h), \gamma(h')\} \in \Gamma$, for some $j < i$. Suppose $|\phi_i(h) - \phi_i(h')| \le 4$. Since $\phi_i$ is an embedding of $H_i$ into $P_n^4$, for every $u \in N_{H_i}(h)$ we have $|\phi_i(u) - \phi_i(h')| \le 8$. From the property (H2) we then obtain $f_i(h') \neq f_i(u)$ and, more generally, $f_i(h') \notin f_i(N_{H_i}(h))$. Therefore, we have $\rho_{f_i(h)}(f_i(h')) \notin \rho_{f_i(h)}(f_i(N_{H_i}(h))) = \Phi_i(h)$, which contradicts the condition (E2). Suppose now that $|\phi_i(h) - \phi_i(h')| > 4$. However, as $\{f_j(h), f_j(h')\}, \{f_i(h), f_i(h')\} \in R_m^4$ we deduce $h' \in B_i(h)$ and so $\{r_i(h), r_i(h')\} \notin R_z^4$, by the property (H4). In any case, we get a contradiction with the assumption that $j, i$ is a witness for the edge $\{\gamma(h), \gamma(h')\}$. This finishes the proof that $\gamma$ is an induced embedding. 
\end{proof}

It remains to show that we can find such homomorphisms. Consider the ordering $h_0, \ldots, h_{n-1}$ of $V(H_i)$ such that $\phi_i(h_t) = t$ for every $t \in [n-1]$. Our construction of $f_i$ and $r_i$ relies on the following observation: suppose we are given a graph $G$ and a function $f \colon V(H_i) \rightarrow V(G)$ such that $W = (f(h_0), \ldots, f(h_{n-1}))$ is a walk in $G$ and $f(h_t) \neq f(h_{t'})$ for every distinct $t, t' \in [n-1]$ with $|t - t'| \le 4$. Then, since $H_i$ is a subgraph of $P_n^4$, $f$ is also a homomorphism of $H_i$ into $G^4$. In other words, if the sequence $W$ forms a walk in $G$ which is `locally' a path, then $f$ is a homomorphism of $H_i$ into $G^4$.

Let $\ell_m = |V(R_m)|$ and $q_m = \lceil \log \ell_m / \log 10 \rceil$. We say that $f_i$ is \emph{well-distributed} if
$$
	|f_i^{-1}(v)| \le \sqrt{n}, \quad \text{ for every } \; v \in V(R_m).
$$ 
With the previous observation in mind, the existence of a well-distributed homomorphism $f_1$ follows from Lemma \ref{lemma:f_walk} applied with $\sigma(t) = \emptyset$ for every $t \in [n-1]$. In particular, let $f \colon [n-1] \rightarrow V(R_m)$ be the mapping obtained by Lemma \ref{lemma:f_walk} and set $f_1(h_t) := f(t)$. From the property (F1) we get
$$
	|f_1^{-1}(v)| \le 40 \lceil n / \ell_m \rceil \le 50 n / m \le \sqrt{n}, 
$$
for every $v \in V(R_m)$. From the property (F3) we infer that $(f_1(h_0), \ldots, f_1(h_{n-1}))$ forms a desired walk. 

Note that we could achieve the same by replacing $R_m$ corresponding to the first coordinate by the $4$-th power of a cycle on $\sqrt{n}$ vertices, in which case $f_1$ is simply defined by going around such a cycle. As this would slightly complicate the condition (E1) and some of the arguments, and only has negligible effect on the number of vertices, we use $R_m$ as stated.

Next, we inductively construct $f_i$ and $r_i$ for $2 \le i \le \Delta$.
 
\begin{claim} 
	Suppose $2 \le i \le \Delta$ and $f_1, \ldots, f_{i-1}$ are well-distributed. Then there exists a well-distributed homomorphism $f_i \colon H_i \rightarrow R_m^4$ which satisfies properties (H1)--(H3).
\end{claim}
\begin{proof}
	Consider the ordering $h_0, \ldots, h_{n-1}$ of $V(H)$ such that $\phi_i(h_t) = t$ for every $t \in [n-1]$. Similarly as in the case of $f_1$, we aim to deduce the existence of $f_i$ by applying Lemma \ref{lemma:f_walk} with suitably defined $\sigma(\cdot)$. In particular, for each $t \in [n-1]$ we set $\sigma(t) = D_1(t) \cup D_3(t)$, where $D_1(t)$ and $D_3(t)$ are chosen such that having $f_i(h_t) \notin D_1(t) \cup D_3(t)$ maintains the properties (H1) and (H3), respectively. 

	Let
	$$
		t_0(t) := (\lfloor t / q_m \rfloor - 1) q_m - 1,
	$$	
	and for each $t \in [n-1]$ let $D_1(t) \subseteq [n-1]$ be defined as follows,
	$$
		D_1(t) = \{0 \le t' \le t_0(t)\; : \; f_1(h_t) = f_1(h_{t'}) \}.
	$$
	Note that $D_1(t)$ is a subset of indices $t'$ for which we require $f_i(h_t) \neq f_i(h_{t'})$, as otherwise the property (H1) becomes violated. From the assumption $|f_1^{-1}(v)| \le \sqrt{n}$ for every $v \in V(R_m)$, we obtain $|D_1(t)| \le \sqrt{n}$. Next, we define a subset $D_3(t) \subseteq [n]$,
	\begin{align*}
		D_3(t) = \{0 \le t' \le t_0(t) \; : \; &\{h_t, h_{t'} \} \notin H \; \text{ and } \\
		&\exists j < i \; \text{ such that } \{f_j(h_t), f_j(h_{t'})\} \in R_m^4\; \}.
	\end{align*}
	The role of $D_3(t)$ is to restrict $B_i(h_t)$ to vertices $h_{t'}$ with $|t - t'| \le 2q_m$. As we will see shortly, this suffices for the property (H3) to hold. From the assumption $|f_j^{-1}(v)| \le \sqrt{n}$ for $1 \le j < i$ we obtain $|D_3(t)| \le \Delta 4d^4 \sqrt{n}$ as follows: there are at most $\Delta$ choices for $j < i$, at most $4d^4$ choices for a neighbour $v'$ of $f_j(h_t)$ in $R^4_m$, and then at most $\sqrt{n}$ choices for $h_{t'} \in f_{j}^{-1}(v')$. Finally, by the choice of $m \le \ell_m$ we get
	$$
		|\sigma(t)| = |D_1(t) \cup D_3(t)| \le 5 \Delta d^4 \sqrt{n} \le \ell_m / (160 d^4).
	$$
	Therefore, we can apply Lemma \ref{lemma:f_walk} with $\sigma$ (note that the first condition is satisfied by the definition of $D_1$ and $D_3$ and the choice of $t_0(t)$). Let $f$ be the obtained function and set $f_i(h_t) := f(t)$ for $t \in [n-1]$. We claim that $f_i$ satisfies the required properties. 

	First, it follows from the property (F3) that $f_i(h_0), \ldots, f_i(h_{n-1})$ is a walk in $R_m$ with $f_i(h_t) \neq f_i(h_{t'})$ for every distinct $t, t' \in [n-1]$ with $|t - t'| \le 8$ (thus the property (H2) holds). Therefore, $f_i$ is a homomorphism of $H_i$ into $R_m^4$. Moreover, from the property (F1) we obtain that $f_i$ is well balanced.

	Next, consider some $t' < t$ such that $f_1(h_t) = f_1(h_{t'})$. If $t' \le t_0(t)$ then $t' \in \sigma(t)$ and from (F2) we get $f_i(h_t) = f(t) \neq f(t') = f_i(h_{t'})$. Otherwise, from the property (F3) we have that 
	$$
		(f(t_0 + 1), \ldots, f(t'), \ldots, f(t))
	$$
	is a path and thus $f_i(h_t) \neq f_i(h_{t'})$. This proves the property (H1).

	It remains to show that the property (H3) holds. Consider some $t \in [n-1]$. Observe that for any $h_{t'} \in B_i(h_t)$ there exist two distinct walks from $h_{t'}$ to $h_t$ in $R_m$, one of length at most $4$ and the other of length at most $|t - t'|$ (and at least $5$). In particular, this implies that $R_m$ contains a cycle of length at most $|t - t'| + 4$ and from the assumption on the girth of $R_m$ we conclude $|t - t'| > \frac{1}{3} \log_{d-1} \ell_m > 2z$, for sufficiently large $m = m(n)$. This proves the second part of the property (H3). 

	To prove the first part of the property (H3), we first observe that the property (F2) implies $\{f_i(h_t), f_i(h_{t'})\} \notin R_m^4$ for every $t' \in D_3(t)$. In particular, if $h_{t'} \in B_i(h)$ and $t' < t$, then $t' \ge t_0(t) + 1$. On the other hand, if $h_{t'} \in B_i(h_t)$ then also $h_t \in B_i(h_{t'})$, and if $t' > t$ then by the same argument we obtain $t \ge t_0(t')$. This further implies $t' < \lfloor t / q_m \rfloor + 2q_m$. To summarise, for every $h_{t'} \in B_i(h_t)$ we have
	\begin{equation} \label{eq:tp_bound}
		t_0(t) < t' < t_0(t) + 3q_m.
	\end{equation}
	Since $|t' - t| > 4$ for $h_{t'} \in B_i(h_t)$, we infer that for any two vertices $h_{t'}, h_{t''} \in B_i(h_t)$ there exist two distinct walks between $h_{t'}$ and $h_{t''}$ in $R_m$, one of length at most $8$ and the other of length at most $|t' - t''|$. As in the previous case, it is easy to see that this implies $R_m$ contains a cycle of length at most $|t' - t''| + 8$. Therefore, by the assumption on the girth of $R_m$ we conclude $|t' - t''| > \frac{1}{3} \log_{d-1} \ell_m$. Together with \eqref{eq:tp_bound}, this gives the following upper bound on the size of $B_i(h_t)$,
	$$
		|B_i(h_t)| < \frac{3q_m}{ \log_{d-1} \ell_m / 3} < d,
	$$ 
	as required. This finishes the proof of the claim.  
\end{proof}

Finally, assuming we have already defined $f_1, \ldots, f_i$ such that (H1)--(H3) holds, we construct a homomorphism $r_i$.

\begin{claim}
	Suppose $2 \le i \le \Delta$ and $f_1, \ldots, f_i$ satisfy properties (H1)--(H3). Then there exists a homomorphism $r_i \colon H_i \rightarrow R_z^4$ which satisfies the property (H4).
\end{claim}
\begin{proof}
	Let $\ell_z = |V(R_z)|$ and $q_z = \lceil \log \ell_z / \log 10 \rceil$, and consider the ordering $h_0, \ldots, h_{n-1}$ of $V(H)$ such that $\phi_i(h_t) = t$ for every $t \in [n-1]$. 

	For each $t \in [n]$, we define $\sigma(t)$ as follows,
	$$
		\sigma(t) = \{ t' \in B_i(h_t) \; : \; t' < t \}.
	$$
	From the choice of $z$ and the first part of the property (H3) of $f_i$ we get $|\sigma(t)| \le d \le \ell_z / (160 d^4)$. From the second part of (H3) we have 
	$$
		t' < t - 2z < t - 2q_z < (\lfloor t / q_z \rfloor - 1) q_z,
	$$
	for every $t' \in \sigma(t)$. Therefore, we can apply Lemma \ref{lemma:f_walk} with $F = R_z$ and $\sigma$ to obtain $f \colon [n-1] \rightarrow V(R_z)$. Set $r_i(h_t) := f(t)$ for $t \in [n-1]$. The same argument as in the proof of the previous claim shows that $r_i$ is a homomorphism. Moreover, since for each $h_{t'} \in B_i(h_t)$ we also have $h_t \in B_i(h_{t'})$, we can assume $t' < t$, and by (F2) we have that $f(t)$ and $f(t')$ are at distance at least $5$ in $R_z$. In other words, we conclude $\{r_i(h_t), r_i(h_{t'})\} \notin R_z^4$ as required. This proves the claim.
\end{proof}

\subsection{The algorithmic aspects of the proof} \label{sec:algo}

The proof of Theorem \ref{thm:main} relies on a decomposition guaranteed
by Theorem \ref{thm:decompose} and the existence of a mapping given by
Lemma \ref{lemma:f_walk}. Note that all other steps needed to construct
the homomorphisms $f_i$ and $r_i$, and therefore an induced embedding of
$H$, can be implemented efficiently (i.e. with polynomial running time)
as they only involve computing sets of vertices which satisfy certain
simple conditions.

The proof of Theorem \ref{thm:decompose} relies on the Edmonds-Gallai
decomposition (e.g. see \cite{lovasz2009matching}) and Hall's
criteria for matchings in bipartite graphs. Since the Edmonds-Gallai
decomposition can be obtained using Edmonds' algorithm for finding a
maximum matching in a general graph and a careful inspection of the
proof of Theorem \ref{thm:decompose} shows that all other steps can be
performed efficiently, computing the desired decomposition can also be
done efficiently. Similarly, the proof of Lemma \ref{lemma:f_walk} is
constructive and, in turn, relies on sets of $q$-expanding vertices
given by Lemma \ref{lemma:expand}. However, the proof of Lemma
\ref{lemma:expand} is algorithmic and gives an efficient deterministic
procedure for computing both a set of sufficiently many 
$q$-expanding vertices and a
set of reachable vertices from each vertex in it. We omit further
details.


\vspace{3mm}
\noindent
{\bf Acknowledgements.} The second author would like to thank 
Nemanja \v Skori\'c for useful discussions.

\bibliographystyle{abbrv}
\bibliography{refind}

\end{document}